\documentclass{amsart}

\usepackage{graphicx}
\usepackage{comment}
\usepackage[T1]{fontenc}
\usepackage{float}

\newtheorem{theorem}{Theorem}[section]
\newtheorem{lemma}[theorem]{Lemma}
\newtheorem*{conjecture}{Conjecture}
\newtheorem{alphatheorem}{Theorem}

\theoremstyle{definition}
\newtheorem{definition}[theorem]{Definition}
\newtheorem{example}[theorem]{Example}

\theoremstyle{remark}

\numberwithin{equation}{section}

\begin{document}

\title{Weak Modularity and $\widetilde{A}_n$ Buildings}

\author{Zachary Munro}
\thanks{This work was partially supported by the grant 346300 for IMPAN from the Simons Foundation and the matching 2015-2019 Polish MNiSW fund.}
\address{Department of Mathematics and Statistics, McGill University, Montreal, QC, Canada}
\curraddr{Burnside Hall 1017, 805 Sherbrooke St W, Montreal, QC H3A 2K6, Canada}
\email{zachary.munro@mail.mcgill.ca}





\begin{abstract}
The $\widetilde{A}_n$ Coxeter groups are known to not be systolic or cocompactly cubulated for $n\geq 3$. We prove that these groups act geometrically on weakly modular graphs, a weak notion of nonpositive curvature generalizing the 1-skeleta of $\mathrm{CAT}(0)$ cube complexes and systolic complexes. To prove weak modularity we describe the canonical emeddings of the 1-skeleta of $\widetilde{A}_n$ Coxeter complexes into the Euclidean spaces $\mathbb{R}^{n+1}$. We also prove weak modularity for buildings of type $\widetilde{A}_3$. 
\end{abstract}

\maketitle

\section*{Introduction}
Coxeter groups, a generalization of discrete reflection groups, are known to satisfy a number of nonpositive curvature properties. In his PhD thesis \cite{moussong}, Moussong proved finitely generated Coxeter groups are $\mathrm{CAT}(0)$. Niblo and Reeves proved in \cite{nibloreeves} that Coxeter groups act properly on $\mathrm{CAT}(0)$ cube complexes. However, there are known obstructions to Coxeter groups acting properly and cocompactly on $\mathrm{CAT}(0)$ cube complexes. The group $\widetilde{A}_2$ is an example of such a group. Simplicial nonpositive curvature conditions are known to hold for some Coxeter groups. In particular, Coxeter groups with all defining coefficients $m_{ij}$ greater than or equal to three are systolic. However, work of Przytycki and Schwer \cite{piotrschwer}; Karrer, Schwer, and Struyve \cite{schwer}; and Wilks \cite{wilks} show the triangle groups 244, 245, and 255 are not systolic. Relevant to this article, the groups $\widetilde{A}_n$ for $n\geq 3$ are also not systolic, following from work of Januszkiewicz and \'Swi\k{a}tkowski \cite{januszkiewicz}.

Weak modularity is a notion of nonpositive curvature for graphs which generalizes the structure of 1-skeleta of $\mathrm{CAT}(0)$ cube complexes and systolic complexes, recently studied in \cite{weakmodularity}. A group $G$ is \textit{weakly modular} if it acts properly and cocompactly by automorphisms on a weakly modular graph. In the absence of proper cocompact actions on $\mathrm{CAT}(0)$ cube complexes and systolic complexes, we prove that the $\widetilde{A}_n$ Coxeter groups are weakly modular. In fact, we first give a description of the 1-skeleta of the $\widetilde{A}_n$ Coxeter complexes in the following theorem. 

\begin{alphatheorem}
\label{1skeleton}
The $1$-skeleton of the $\widetilde{A}_n$ Coxeter complex can be described as follows. The vertices are of the form $(x_0,\ldots,x_n)\in \mathbb{Z}^{n+1}$ where $x_0+\ldots+x_n=0$ and $x_i-x_j\equiv 0\mod (n+1)$ for all $i$ and $j$. Two vertices are adjacent if and only if they differ by a vector of the form $(e_0,\ldots,e_n)\in \mathbb{Z}^{n+1}$ where $e_0+\ldots+e_n=0$, $e_i-e_j\equiv 0\mod (n+1)\mathbb{Z}$ and $\max_i e_i-\min_je_j = n+1$.
\end{alphatheorem}

Coxeter groups are known to act properly and cocompactly on their associated Coxeter complex if the subgroups generated by all but one of the generators are finite, as is the case with $\widetilde{A}_n$. In particular, they act properly and cocompactly on the 1-skeleta of their associated complex. So we prove the following theorem, from which it follows that $\widetilde{A}_n$ Coxeter groups are weakly modular.

\begin{alphatheorem}
\label{wmcomplexes} 
The 1-skeleta of $\widetilde A_n$ Coxeter complexes are weakly modular.
\end{alphatheorem}

Finally, in dimension $n=3$, we extend the above theorem to buildings, simplicial complexes built from Coxeter complexes of some given Coxeter group. That is, we prove the following. 

\begin{alphatheorem}
\label{wmbuildings} 
The 1-skeleta of buildings of type $\widetilde{A}_3$ are weakly modular.
\end{alphatheorem}

We believe that this result is true for all $n$, but it is unclear how the given proof generalizes or if different techniques are required. 

\begin{conjecture}
The 1-skeleta of buildings of type $\widetilde{A}_n$ are weakly modular.
\end{conjecture}

\smallskip
\noindent \textbf{Acknowledgements.}
I would like to thank my advisers Piotr Przytycki and Marcin Sabok for their generosity and patience in meeting with me over the duration of this work. This article would not exist without their support and guidance. I would also like to thank my parents, Rebecca and Eric, and my brother, Lucas, for their encouragement of my academic pursuits. I am lucky to have you all as family. 

\section{Definitions}
\subsection{Coxeter Complexes and Buildings}
\begin{definition}
A \textit{Coxeter group} $W$ is a group having a presentation of the form\linebreak $\langle s_1,\ldots,s_n|s_i^2, (s_is_j)^{m_{ij}}\ \forall i\neq j \rangle$, where $m_{ij}=m_{ji}\in\mathbb{N}_{\geq 2}\cup \{\infty\}$, and $m_{ij}=\infty$ means we omit a relation for the order of $s_is_j$. 
\end{definition}
Suppose we have a Coxeter group $W$. There are several ways to encapsulate the presentation for a Coxeter group. One way is via an $n\times n$ matrix with entries $(m_{ij})$, where $m_{ii}=1\ \forall i$, called the \textit{Coxeter matrix} of the group $W$. Another way is via a complete edge-labeled graph, with vertex set the generators $\{s_1,\ldots,s_n\}$ and edges labeled by $m_{ij}$ with endpoints $s_i$, $s_j$. This graph is called the \textit{Coxeter diagram} of the group $W$. It is a convention that edges corresponding to $m_{ij}=2$ are omitted and edges corresponding to $m_{ij}=3$ are drawn unlabeled. We can now introduce our central object of study. 

\begin{definition}
For $n\geq 3$, the $\widetilde{A}_n$ Coxeter group is the Coxeter group with Coxeter diagram an unlabeled $(n+1)$-cycle. 
\end{definition}
The presentation which exhibits a group $W$ as a Coxeter group is not necessarily unique. 
Thus, it is helpful to fix a particular presentation of the group, which is why one often deals with Coxeter systems.

\begin{definition}
A \textit{Coxeter system} is a pair $(W,S)$ where $W$ is a Coxeter group and $S$ is a generating set of $W$ coming from some presentation exhibiting $W$ as a Coxeter group.
\end{definition}
Since Coxeter groups are meant to generalize discrete reflection groups, it should not be surprising that it is possible to ``artificially'' construct a space on which a given Coxeter group acts by ``reflections'', in an appropriate sense. 

\begin{definition}
The \textit{Coxeter complex} $\Sigma(W,S)$ associated to a Coxeter system $(W,S)$ is an $(|S|-1)$-dimensional simplicial complex obtained by identifying pairs of faces in a disjoint collection of $(|S|-1)$-simplices. For each $w\in W$, there is an associated $(|S|-1)$-simplex with vertices labeled by $S$ called a \textit{chamber}. We glue simplices $w$ and $w'$ along their faces opposite to vertex $s\in S$ (respecting the labels on the identified faces) if and only if $w=w's$. 
\end{definition}

Notice the vertices of each simplex are labeled by the generators $S$ and all identifications of faces respect the labels. Thus, vertices $v$ in a Coxeter complex have a well-defined label in $S$, called the \textit{type} of $v$, induced by the labeling on the vertices of the simplices. We will sometimes refer to a vertex as an ``$s$ vertex'' to mean a vertex of type $s$. 

There is a nice relationship between links in a Coxeter complex and the Coxeter diagram. Let $D$ be the Coxeter diagram of a group $(W,S)$, so that vertices of $D$ can be identified with~$S$, and let $v\in \Sigma(W,S)$ be a vertex with label $s$. An analysis of the Tits representation of a Coxeter group, e.g.\ in \cite{ronan}, shows that the generators $S-s$ generate (in $W$) a Coxeter group with diagram $D-s$, as one might expect. Thus the link $lk(v)$ is naturally identified with the Coxeter complex associated with $D-s$. In particular, the link of a vertex in the $\widetilde{A}_n$ Coxeter complex is isomorphic to the $A_n$ Coxeter complex. 

\begin{definition}
For $n\geq 2$, the $A_n$ Coxeter group is the Coxeter group with Coxeter diagram a path with $n$ vertices. 
\end{definition}
Coxeter groups can be arranged into larger simplicial complexes called buildings, the theory of which is largely due to Jacques Tits. Ultimately, we will extend our proof of weak modularity to certain buildings.

\begin{definition}
A \textit{building} $\mathcal{B}$ with $\Sigma$ apartments is a simplicial complex which is the union of subcomplexes $\mathcal{B}=\bigcup_iA_i$ called \textit{apartments}, each isomorphic to the Coxeter complex $\Sigma$ such that: 
\begin{enumerate}
    \item Any two simplices are contained in an apartment.
    \item For any two apartments $A_i$, $A_j$ there exists an isomorphism $A_i\to A_j$ fixing $A_i\cap A_j$, where the intersection can potentially be empty.
\end{enumerate} 
\end{definition}
If $\Sigma$ is the Coxeter complex of a Coxeter group $W$, then we call $\mathcal{B}$ a \textit{building of type $W$}. Though the general construction of a Coxeter complex is combinatorial, there are subclasses of Coxeter groups whose associated Coxeter complex embeds in Euclidean space, in hyperbolic space, or in a sphere.

Discussion in \cite[Ch 3.1]{brown} shows how the Coxeter complex of $A_n$ can be identified with a sphere. Similarly, discussion in \cite[Ch 10.2.1]{brown} shows how $\widetilde{A}_n$ can be given a Euclidean space structure. That is, the Coxeter complexes can be thought of as simplicial decompositions of spheres and Euclidean spaces so that the generators of the respective Coxeter groups act by reflections through the faces of some given simplex. Reflections through adjacent faces in the given simplex have an order prescribed in the Coxeter group presentation, and from this we can deduce the dihedral angle between two faces. 

Given any two points in a building $x,y\in \mathcal{B}$, there is an apartment $A$ containing $x$ and $y$. If we define $d(x,y)$ to be the distance between $x$ and $y$ in $A$, then we get a well-defined metric on $\mathcal{B}$ which is $\mathrm{CAT}(0)$ (resp.\ $\mathrm{CAT}(1)$) if the apartments are isometric to Euclidean space (resp. some sphere) \cite[Thm 10A.4]{bridson}. With this metric on $\mathcal{B}$, apartments are convex \cite[Thm 10A.5]{bridson}. Note also that the link of a vertex in a building $\mathcal{B}$ of type $\widetilde{A}_n$ is a building of type $A_n$. The $\mathrm{CAT}(1)$ metric that links inherit from $\mathcal{B}$ as a $\mathrm{CAT}(0)$ space is the same as the $\mathrm{CAT}(1)$ metric they inherit as buildings of type $\mathrm{CAT}(0)$, so no distinction is necessary.

\subsection{Weak Modularity}\hfill\\
We now introduce a nonpostive curvature property of graphs which generalizes the structure of 1-skeleta of $\mathrm{CAT}(0)$ cube complexes and systolic complexes. 
\begin{definition}
\label{weakly modular}
A graph $G=(V,E)$ is said to be \textit{weakly modular} if it satisfies the following two conditions for any vertex $v\in V$. We let $d$ denote the path metric on $G$. 
\begin{enumerate}
    \item The \textit{triangle condition}: if $u,v,w\in V$ such that $(u,w)\in E$ and $d(u,v)=n=d(w,v)$, then there exists a vertex $t\in V$ such that $(u,t),(w,t)\in E$ and $d(t,v)=n-1$
    \item The \textit{quadrangle condition}: if $u,v,w,s\in V$ such that $(s,u),(s,w)\in E$, $d(u,v)=d(w,v)=n$ and $d(s,v)=n+1$, then there exists a vertex $t\in V$ such that $(u,t),(w,t)\in E$ and $d(t,v)=n-1$
\end{enumerate}

\end{definition}
A graph is $\textit{locally weakly modular}$ if it satisfies the triangle property for edges $uw$ distance two from $v$ in definition \ref{weakly modular}.i and the quadrangle property for vertices $s$ at distance three from $v$ in definition \ref{weakly modular}.ii. The triangle-square completion $X_{\triangle \square}(G)$ of a graph $G$ is the 2-complex with 1-skeleton $G$ and with 2-cells attached to every 3-cycle and 4-cycle. Weak modularity is studied in \cite{weakmodularity}. In particular, the authors prove the following local-to-global theorem.  

\begin{theorem}[Local-to-Global {\cite[Thm 3.1]{weakmodularity}}]
\label{localglobal}
A graph $G$ is a weakly modular graph if and only if $G$ is a locally weakly modular graph whose triangle-square complex $X_{\triangle \square}(G)$ is simply connected.
\end{theorem}
In particular, since buildings are simply-connected simplicial complexes, it suffices to prove that the 1-skeleton of a building is locally weakly modular to prove it is weakly modular. 

\section{Constructing $\widetilde{A}_n$-complexes} 
\begin{definition}[\cite{brown}]
A group $W$ of isometries of a finite-dimensional real vector space $V$ is an \textit{affine reflection group} if there is a set of hyperplanes $\mathcal{H}$ such that
\begin{enumerate}
    \item $W$ is generated by the reflections $\{s_H :H\in\mathcal{H}\}$ where $s_H$ is the orthogonal reflection through the hyperplane $H$
    \item $\mathcal{H}$ is $W$-invariant
    \item $\mathcal{H}$ is locally finite
\end{enumerate}
\end{definition}
Recall that a hyperplane $H\subset V$ is the set of points $x\in V$ which satisfy an equation of the form $f(x)=c$, where $f\in V^*$ and $c\in \mathbb{R}$. Equivalently $H=S+x$ where $S\subset V$ is an $(n-1)$-dimensional subspace and $x\in V$. $f$ is the dual of a normal vector of $S$, and the sets of points which satisfy $f>c$ and $f<c$ are exactly the two disjoint open half-spaces in $V-H$.

We select for each $H\in \mathcal{H}$ a defining equation $f_H=c_H$. Our hyperplanes $\mathcal{H}$ partition our vector space $V$ into sets called $\textit{cells}$, where a cell $A\subset V$ is a nonempty set defined by satisfying one of $f_H=c_H$, $f_H>c_H$, or $f_H<c_H$ for each $H\in\mathcal{H}$. Each cell $A$ can be identified by a sequence $\sigma(A)_H\in \{-,0,+\}$ indexed by $\mathcal{H}$, where $\sigma(A)_H$ being $-$, $0$, or $+$ corresponds to $A$ satisfying $f_H<c_H$, $f_H=c_H$, or $f_H>c_H$, respectively. With such a sequence associated to each cell, we can define a partial order by declaring $B\leq A$ if and only if $\forall H\in\mathcal{H}$ either $\sigma_H(B)=0$ or $\sigma_H(B)=\sigma_H(A)$. Equivalently, $B\leq A$ if and only if $B\subset \overline{A}$, the closure of $A$. In case $B\leq A$ we say $B$ is a $\textit{face}$ of $A$. With its poset structure, we denote the set of cells by $\Sigma(\mathcal{H})$. 

From a poset $\Sigma({\mathcal{H}})$ we can construct a simplicial complex as follows. The vertices of the complex are in 1-1 correspondence with the lowest dimensional faces of $\Sigma(\mathcal{H})$. Then, a collection of vertices span a simplex if and only if the corresponding faces have a common greater element in $\Sigma(\mathcal{H})$. This simplicial complex is denoted $\Sigma(W,V)$, for which we have the following result. 

\begin{theorem}[\cite{brown}] 
\label{construction}
Let $W$ be an affine reflection group with associated hyperplanes $\mathcal{H}$ in an $n$-dimensional real vector space $V$. Let $C$ be an $n$-dimensional cell and let $S$ be the set of reflections through the $(n-1)$-dimensional faces of $C$. Then $(W,S)$ is a Coxeter system and $\Sigma(W,V)\cong \Sigma(W,S)$.  
\end{theorem}
We aim to study the 1-skeleton of the Coxeter complex of $\widetilde{A}_n$. Our approach first involves embedding the 1-skeleton in Euclidean space $\mathbb{R}^{n+1}$ by realizing $\widetilde{A}_n$ as a Euclidean reflection group. From this embedding we obtain a particular description of the 1-skeleton, from which a suitable analysis allows us to prove weak modularity. 

To realize the 1-skeleton of a $\widetilde{A}_n$-Coxeter complex inside Euclidean space, we make use of Theorem $\ref{construction}$ on Euclidean reflection groups. We proceed by giving a detailed construction of $\widetilde{A}_2$, which we then generalize to the construction of all $\widetilde{A}_n$. It was pointed out to the author by Jon McCammond that our description of the 1-skeleta of the $\widetilde{A}_n$ Coxeter complexes can be found without proof in Conway and Sloane's text~\cite{conway}. 

\begin{example}[Constructing the $\widetilde{A}_2$ 1-skeleton]
Inside $\mathbb{R}^3$ with coordinates $(x,y,z)$ we consider hyperplanes $\mathcal{H}$ of the form $\{x-y=c\}$, $\{y-z=c\}$, and $\{x-z=c\}$ for every $c\in 3\mathbb{Z}$. 

\textbf{Step 1.} \textit{The group $W$ generated by reflections through hyperplanes in $\mathcal{H}$ is an affine reflection group, and $W\cong \widetilde{A}_2$}. 

The reflection through a hyperplane of the form $H=\{x-y=c\}$ is given by $s_H(x,y,z)=(y+c,x-c,z)$, with similar formulas for reflecting through other hyperplanes. Such a reflection sends the collection of hyperplanes $\mathcal{H}$ to itself, and thus $\mathcal{H}$ is $W$-invariant, since $W$ is generated by these reflections. Next, note that the ball of radius $1/2$ centered at any point $(x,y,z)\in\mathbb{R}^3$ can intersect at most three hyperplanes in $\mathcal{H}$. Indeed, if we have two close points $d((x_1,y_1,z_1),(x_2,y_2,z_2))<1$ and the first point lies on a hyperplane, for example $x_1-y_1=c\in 3\mathbb{Z}$, then since $|(x_1-y_1)-(x_2-y_2)|\leq |x_1-y_1|+|x_2-y_2|\leq 2$ it is impossible that the second point lies on a different, parallel hyperplane, i.e. it is impossible that $x_2-y_2\in 3\mathbb{Z}$ but $x_2-y_2\neq c$. In particular, since a ball of radius $1/2$ has diameter 1, it can intersect at most one hyperplane in each parallelism class. Thus our collection of hyperplanes $\mathcal{H}$ is locally finite and $W$ is an affine reflection group. 

Next, we show that $W\cong \widetilde{A}_2$. We claim the set $\{x<y<z<x+3\}\subset \mathbb{R}^3$ is a 3-dimensional cell. Indeed, the hyperplanes $\mathcal{H}$ are divided into three parallelism classes, $x-y=3n$, $x-z=3n$, and $y-z=3n$ for $n\in \mathbb{Z}$. Our cell can be defined by the inequalities $x<y<x+3$, $x<z<x+3$ and $y<z<y+3$, which places it in between adjacent pairs of hyperplanes in each of the parallelism classes. Thus $\{x<y<z<x+3\}$ is a 3-dimensional cell. Let $s_{H_i}$ denote the reflection across $H_i$ where $H_1=\{x-y=0\}$, $H_2=\{z-x=3\}$, and $H_3=\{y-z=0\}$. One checks that the composition $s_{H_2}s_{H_1}$ is given by $(x,y,z)\mapsto (y,z+3,x-3)$, from which one verifies that $(s_{H_2}s_{H_1})^3=id$. A similar computation shows that $s_{H_i}s_{H_j}$ has order three for each $i\neq j$. By Theorem $\ref{construction}$, $(W,\{s_{H_i}\})$ is a Coxeter system, and the computation we just performed shows $W\cong \widetilde{A}_2$. 

\textbf{Step 2.} \textit{The vertex set is in bijection with those points $(x,y,z)$ such that $x+y+z=0$ and the difference of any two coordinates is in $3\mathbb{Z}$. Two vertices are joined by an edge if and only if their difference is one of $\pm(2,-1,-1)$, $\pm(-1,2,-1)$, or $\pm(-1,-1,2)$.}

Since each hyperplane of $\mathcal{H}$ is orthogonal to the hyperplane $x+y+z=0$, we can restrict our attention to this isometrically embedded copy of $\mathbb{R}^2$ and replace $\mathcal{H}$ with its restriction to $x+y+z=0$. The advantage of this is that finding minimal cells is easier. In particular, if we find a 0-dimensional cell, we know it must be minimal and thus correspond to a vertex.  

It is thus clear that each point $(x,y,z)$ as described above is a minimal cell. It is a cell because can be defined as the intersection of three hyperplanes, one from each parallelism class, and minimality is clear since it is 0-dimensional. Also, if a point $(x,y,z)$ is not of the described form, then there exists a coordinate whose difference with any other coordinate is not in $3\mathbb{Z}$. Say, $x-y,x-z\not\equiv 0\mod3$. Then $(x,y,z)+(\epsilon,-\epsilon/2,-\epsilon/2)$ does not cross any hyperplane for all $\epsilon>0$ sufficiently small. Thus $(x,y,z)$ lies in a higher dimensional cell and is not a vertex.

We show that vertices which differ by $(2,-1,-1)$ are joined by an edge. The other cases follow the same argument. If $(x,y,z)$ is a vertex, then the set of points $\{(x,y,z)+t(2,-1,-1):t\in(0,1)\}$ is a cell, from which it follows these two vertices are joined by an edge. Indeed, if $(x,y,z)$ was defined by $x-y=3n_1$, $x-z=3n_2$, and $y-z=3n_3$, then $\{(x,y,z)+t(2,-1,-1):t\in(0,1)\}$ is defined by $3n_1<x-y<3(n_1+1)$, $3n_2<x-z<3(n_2+1)$, and $y-z=3n_3$, showing it is a cell. Since $(x,y,z)$ and $(x,y,z)+(2,-1,-1)$ are contained in the closure of this 1-dimensional cell, they are joined by an edge in the Coxeter complex. 

Next, we show vertices are joined by an edge only if their difference is of the form described above. Let $(e_1, e_2, e_3)=(x_2,y_2,z_2)-(x_1,y_1,z_1)$ be the difference between two distinct vertices. Note that $(e_1,e_2,e_3)$ must have integer coordinates which sum to zero, and the difference between any two coordinates must lie in $3\mathbb{Z}$. Hence the only way the difference between two vertices is not one of the values described above is if there exist coordinates whose difference is greater than three. Say, $e_1-e_2>3$. But then the line segment joining the two vertices crosses a hyperplane. Indeed, if $x_1-y_1=3n$, then $\{(x_1,y_1,z_1)+t(e_1,e_2,e_3):t\in(0,1)\}$ lies on both sides of the hyperplane $x-y=3(n+1)$. Thus the vertices are not contained in the closure of a higher dimensional cell and are not joined by an edge in the Coxeter complex. 
\end{example}

\begin{proof}[Proof of Theorem \ref{1skeleton}]
The proof of the general case follows the same outline as the above example, with minor modifications.

\textbf{Step 1.} 
The hyperplanes $\mathcal{H}$ in $\mathbb{R}^{n+1}$ are of the form $\{x_i-x_j=c\}$ for all coordinates $i\neq j$ and $c\in (n+1)\mathbb{Z}$. An analogous computation to the $\widetilde{A}_2$ case shows that reflections through any hyperplane preserve the collection of hyperplanes $\mathcal{H}$. Similarly, a ball of radius $1/2$ intersects at most one hyperplane in each parallelism class, of which there are ${n+1 \choose 2}$-many, each parallelism class corresponding to different choices of $i$ and $j$ in $\{x_i-x_j=c\}$.

To verify that the affine reflection group $W$ is really $\widetilde{A}_n$, we first select a top-dimensional cell $\{x_0<x_1<\ldots<x_n<x_0+(n+1)\}$. Verifying that this is a cell is analogous to the $\widetilde{A}_2$ case. The cell has faces contained in the hyperplanes $\{x_i-x_{i+1}<0\}$ for $i=0,\ldots,n-1$ and $\{x_n-x_0<n+1\}$, and for each of these hyperplanes, there are two others whose coordinates overlap. For example, the hyperplane $\{x_n-x_0<n+1\}$ has coordinates overlapping with the hyperplanes $\{x_{n-1}-x_n<0\}$ and $\{x_0-x_1<0\}$. A computation shows that the composition of reflections through hyperplanes with overlapping coordinates has order 3, and reflection through hyperplanes with disjoint coordinates has order 2. By Theorem~$\ref{construction}$, we get $W\cong \widetilde{A}_n$.

As before, we proceed by restricting our attention to the $W$-invariant subspace $\{x_0+\ldots+x_n=0\}$, which is orthogonal to all hyperplanes in $\mathcal{H}$.

\textbf{Step 2.}
\textit{The vertices are exactly those points $(x_0,\ldots,x_n)$ which satisfy $x_0+\ldots+x_n=0$ such that $x_i-x_j\equiv 0\mod (n+1)$ for all coordinates $i$ and $j$. Two vertices are joined by an edge if and only if their difference $(e_0,\ldots,e_n)$ satisfies $e_0+\ldots+e_n=0$, $e_i-e_j\equiv 0 \mod (n+1)\mathbb{Z}$, and $\max_i e_i - \min_je_j = n+1$.}

It should be clear that the claimed vertices are minimal cells, which lie on a hyperplane in each parallelism class of hyperplanes. We only need to show no other points are vertices.

For a point $(x_0,\ldots,x_n)$, we split the set of coordinates $\{0,\ldots,n\}$ into equivalence classes where $i\sim j$ if and only if $x_i-x_j\equiv 0\mod (n+1)$. For a point $(x_0,\ldots,x_n)$ to fail our above conditions, there must be at least two distinct equivalence classes $S$ and $P$. For all $\varepsilon >0$ sufficiently small, if we add $\varepsilon/|S|$ and $-\varepsilon/|P|$ to the coordinates in $S$ and $P$, respectively, then we do not change the hyperplane sign sequence of $(x_0,\ldots,x_n)$, which shows that our point is contained in a higher dimensional cell and is thus not a vertex.

Suppose vertices $x=(x_0,\ldots,x_n)$ and $y=(y_0,\ldots, y_n)$ differ by a vector $(e_0,\ldots,e_n)$ which satisfies $e_0+\ldots+e_n=0$, $e_i-e_j\equiv 0\mod (n+1)$ for all coordinates $i$ and~$j$, and $\max_i e_i-\min_j e_j=n+1$. The conditions on $(e_0,\ldots,e_n)$ imply that there are exactly two coordinate values, one positive and one negative, and the difference of the two is $n+1$. We can modify the defining hyperplane equalities of the vertex $x$ to get a set of defining hyperplane inequalities for $\{x+t(e_0,\ldots,e_n):t\in(0,1)\}$. For those coordinates $i$, $j$ such that $e_i-e_j=0$, we make no change to the hyperplane equality $x_i-x_j=3n$ defining $x$. For those coordinates $i$, $j$ such that $e_i-e_j=n+1$, we change the defining hyperplane equality $x_i-x_j=3n$ of $x$ to an inequality $3n<x_i-x_j<3(n+1)$. As in the $\widetilde{A}_2$ case, these changes give us a defining set of inequalities for $\{x+t(e_0,\ldots,e_n):t\in(0,1)\}$. Thus the open segment $\{x+t(e_0,\ldots,e_n):t\in(0,1)\}$ is a cell whose closure contains $x$, $y$, and so $x$, $y$ are joined by an edge.  

Proving all differences of vertices $(e_0,\ldots,e_n)$ not satisfying the described conditions do not correspond to edges is a direct generalization of the $\widetilde{A}_2$ example. Suppose $(e_0,\ldots,e_n)$ is the difference of vertices $x$ and $y$ not satisfying the above conditions. Then there are coordinates $i$, $j$ such that $e_i-e_j>n+1$, and so the line segment joining $x$, $y$ crosses a hyperplane and $x$, $y$ are not joined by an edge in the Coxeter complex. 
\end{proof}

\section{Weak Modularity of $\widetilde{A}_n$-complexes}

The conditions $e_1+\ldots+e_n=0$, $e_i-e_j\equiv 0\mod (n+1)$, and $\max_i e_i-\min_je_j=n+1$ imply that each edge can only have two distinct coordinate values, one positive and one negative with absolute difference $n+1$. Thus, an edge is fully determined by which coordinates are positive and negative. 
\begin{example}
The edges of $\widetilde A_5$ are of the form $\pm (5,-1,-1,-1,-1,-1)$, \\ $\pm (4,4,-2,-2,-2,-2)$, and $\pm(3,3,3,-3,-3,-3)$, up to permutation of the coordinates. These edges can be determined from just the signs of their coordinates. For example, $(+,-,-,+,-,+)$ corresponds to $(3,-3,-3,3,-3,3)$ and $(-,-,+,+,+,+)$ corresponds to $(-4,-4,2,2,2,2)$.  
\end{example}
Any talk of distance will refer to the path metric on the 1-skeleton of $\widetilde A_n$, and we define the \textit{height} of a vertex to be the distance from the origin. 
\begin{lemma}
The height of a vertex $(x_0,\ldots,x_n)\in \widetilde A_n\subset \mathbb{R}^{n+1}$ is equal to $(\max_i x_i-\min_j x_j)/(n+1)$.
\end{lemma}
\begin{proof}
Travelling along an incident edge to $(x_0,\ldots,x_n)$ changes $(\max_ix_i-\min_jx_j)$ by $-(n+1)$, $+(n+1)$, or 0. It thus remains to show there is always an edge which decreases $(\max_ix_i-\min_jx_j)$. 

Suppose $x_{i_1},\ldots,x_{i_k}$ are all the maximal coordinates and $x_{j_1},\ldots,x_{j_r}$ are all the minimal coordinates of $(x_0,\ldots,x_n)$. Traveling an edge with positive values in the minimal coordinates and negative values in the maximal coordinates decreases $(\max_ix_i-\min_jx_j)$.
\end{proof}

Next, we define a useful bookkeeping tool, which allows us to easily compute height and check adjacency between vertices.

\begin{definition}
A \textit{ladder} is an equivalence class of functions $\{0,\ldots,n\}\to \mathbb{Z}$ identified up to translation (i.e., addition of an integer). The \textit{ladder associated to a vertex} $x=(x_0,\ldots,x_n)$ is the equivalence class of functions containing $L_x:~\{0,\ldots,n\}\to~\mathbb{Z}:i\mapsto (x_i-x_0)/(n+1)$.
\end{definition}

The intuition is that we are placing the coordinates of $(x_0,\ldots,x_n)$ on rungs of an infinite ladder, where stepping up/down a rung corresponds to an increase/decrease of $(n+1)$. We can picture the situation by thinking of the coordinates $0,\ldots,n$ as labeled balls sitting on different levels of a ladder. 

\begin{example}
The vertex $(10,10,-5,-5,-10)\in \widetilde A_4$ has a ladder which can be represented as in Figure \ref{poop}.

\begin{figure}[h]
    \centering
    \includegraphics[height=6cm]{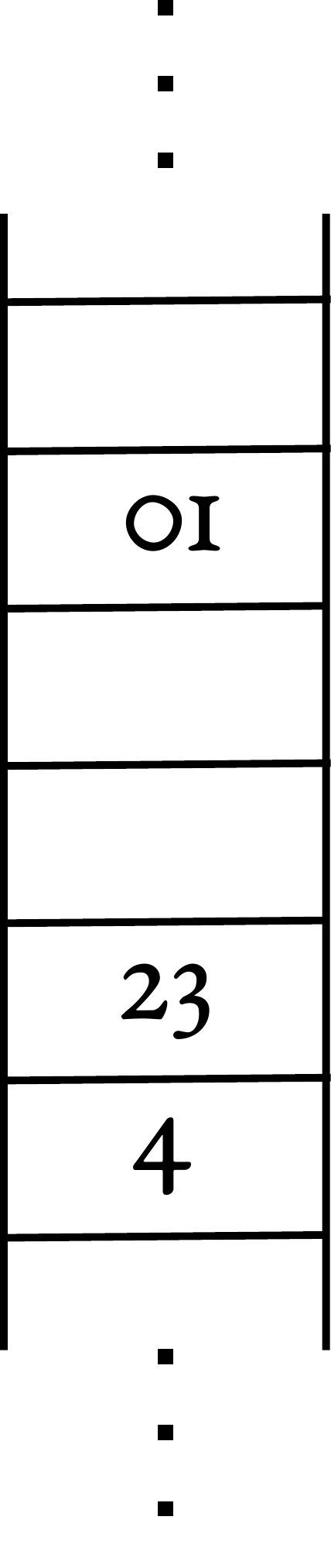}
    \caption{A ladder of a vertex}
    \label{poop}
\end{figure}

Next, we consider how traversing an edge affects the ladder of a vertex. If an edge~$e$ has positive coordinates $i_1,\ldots,i_k$, with remaining coordinates negative, then $L_{x+e}(i_j)=L_X(i_j)+1$ for $j=1,\ldots,k$ and $L_{x+e}(i)=L_x(i)$ for all remaining coordinates $i$. 

In our example above, if we traverse the edge $(4,-2,4,-2,-2)$, then we imagine our ladder changing as indicated in Figure \ref{edge}.
\begin{figure}[h]
\centering
\includegraphics[height=6cm]{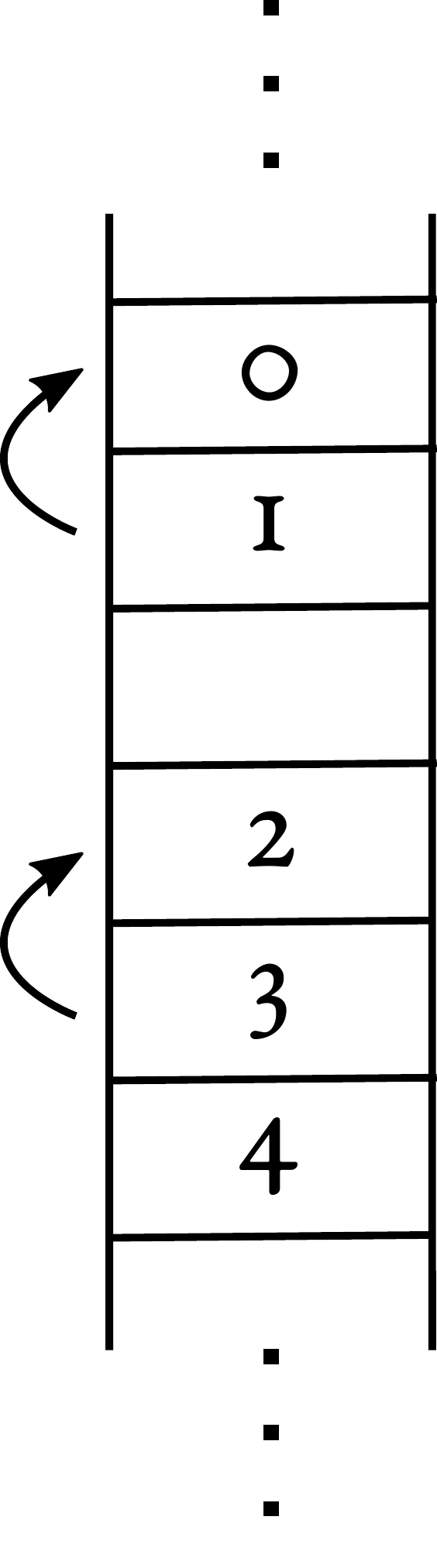}
\caption{Traversing an edge}
\label{edge}
\end{figure}
\end{example}
Thus, traversing any particular edge amounts to selecting a subset $S\subset \{0,\ldots,n\}$ and shifting those coordinates up one rung of the ladder. Also note that the height of a vertex is the distance in the ladder between the highest and lowest coordinates. Thus, traversing an edge decreases height when $S$ contains all balls on the bottom rung and none on the top.
\begin{proof}[Proof of Theorem \ref{wmcomplexes}]
We first prove the triangle property. If $A$ and $B$ are adjacent vertices both of height $n$, then there is some subset $S_{AB}$ of the coordinates so that shifting $S_{AB}$ up one rung moves from $A$ to $B$. That $A$ and $B$ are the same height means that $S_{AB}$ either contains (i) not all of the bottom rung and none of the top rung or (ii) all of the bottom and some of the top rung. After potentially interchanging $A$ and $B$, we can suppose $S_{AB}$ contains not all of the bottom rung and none of the top rung. Let $S_{\text{min}}$ denote the set of coordinates in the bottom rung of $A$ not included in $S_{AB}$, and let $C$ be the vertex reached from $B$ by increasing $S_{\text{min}}$ by a rung -- so $B$ and $C$ are adjacent. Note that $A$ is adjacent to $C$ as well by traveling along the edge $S_{\text{min}}\cup S_{AB}$. Thus $C$ is a height $n-1$ vertex adjacent to $A$ and $B$. 

Next we prove the quadrangle property. Let $X$ be a vertex of height $n+1$ with adjacent vertices $Y$ and $Z$ of height $n$. The subsets of coordinates increased from $X$ to $Y$ and $X$ to $Z$, $S_{XY}$ and $S_{XZ}$, respectively, contain all coordinates on the bottom rung and none of the top rung. Increase the bottom coordinates of the ladder $L_X$ by two rungs and the middle coordinates (those not on top or bottom) by one rung to get a ladder for $W$, a new vertex, pictured in Figure \ref{quadrangle}. Then $W$ has height $n-1$ and both $Y$ and $Z$ are adjacent to $W$.
\end{proof}

\begin{figure}[h]
\centering
\includegraphics[height=7cm]{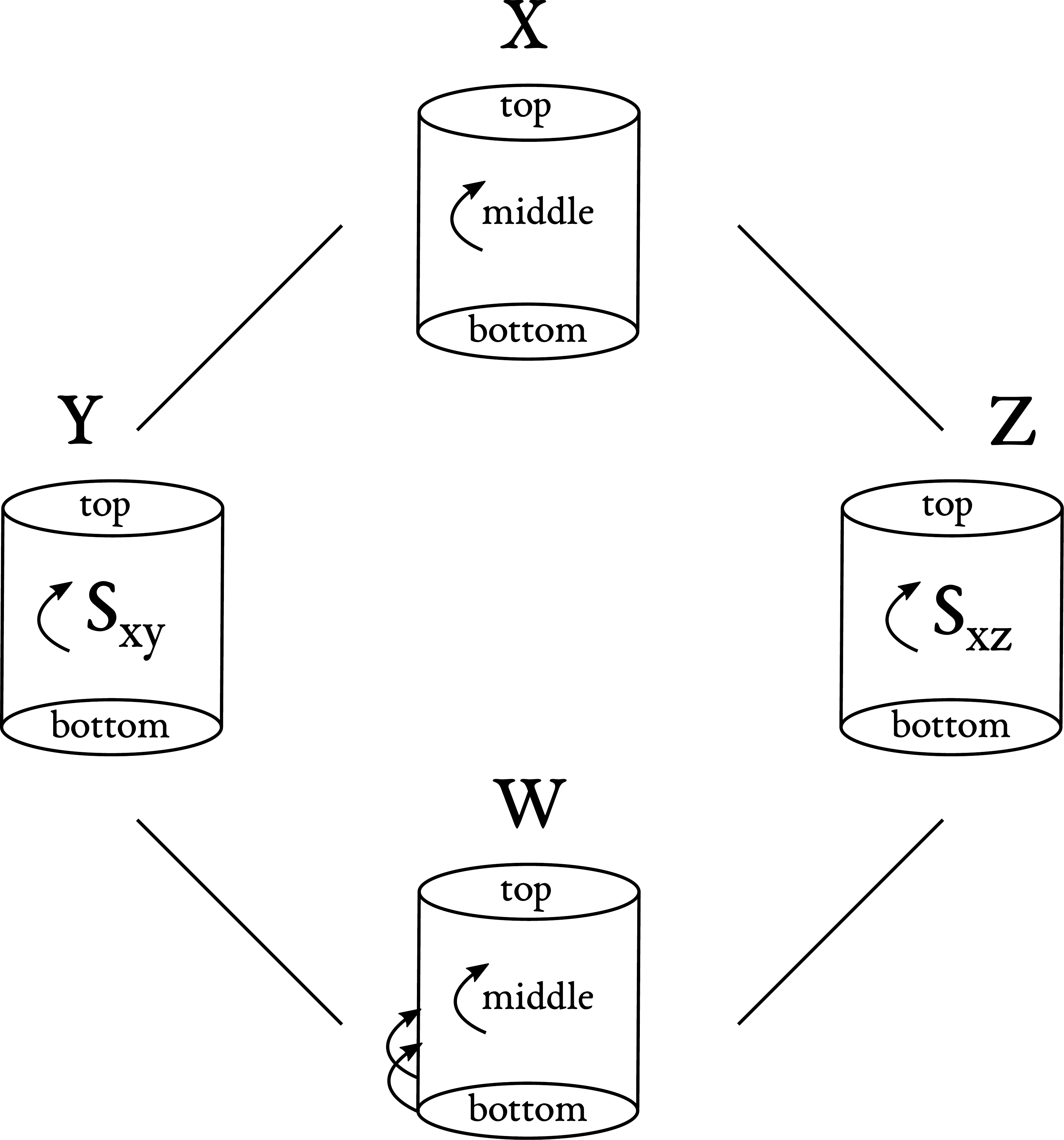}
\caption{The quadrangle property}
\label{quadrangle}
\end{figure}

\section{Weak Modularity of $\widetilde{A}_3$-Buildings}
In this section, we prove that buildings of type $\widetilde{A}_3$ are weakly modular. However, we first prove a lemma needed for the theorem. For the remainder of the section, we fix generators for the Coxeter group $\widetilde{A}_3$ as in Figure \ref{diagram}. 

\begin{figure}[h]
\centering
\includegraphics[height=2cm]{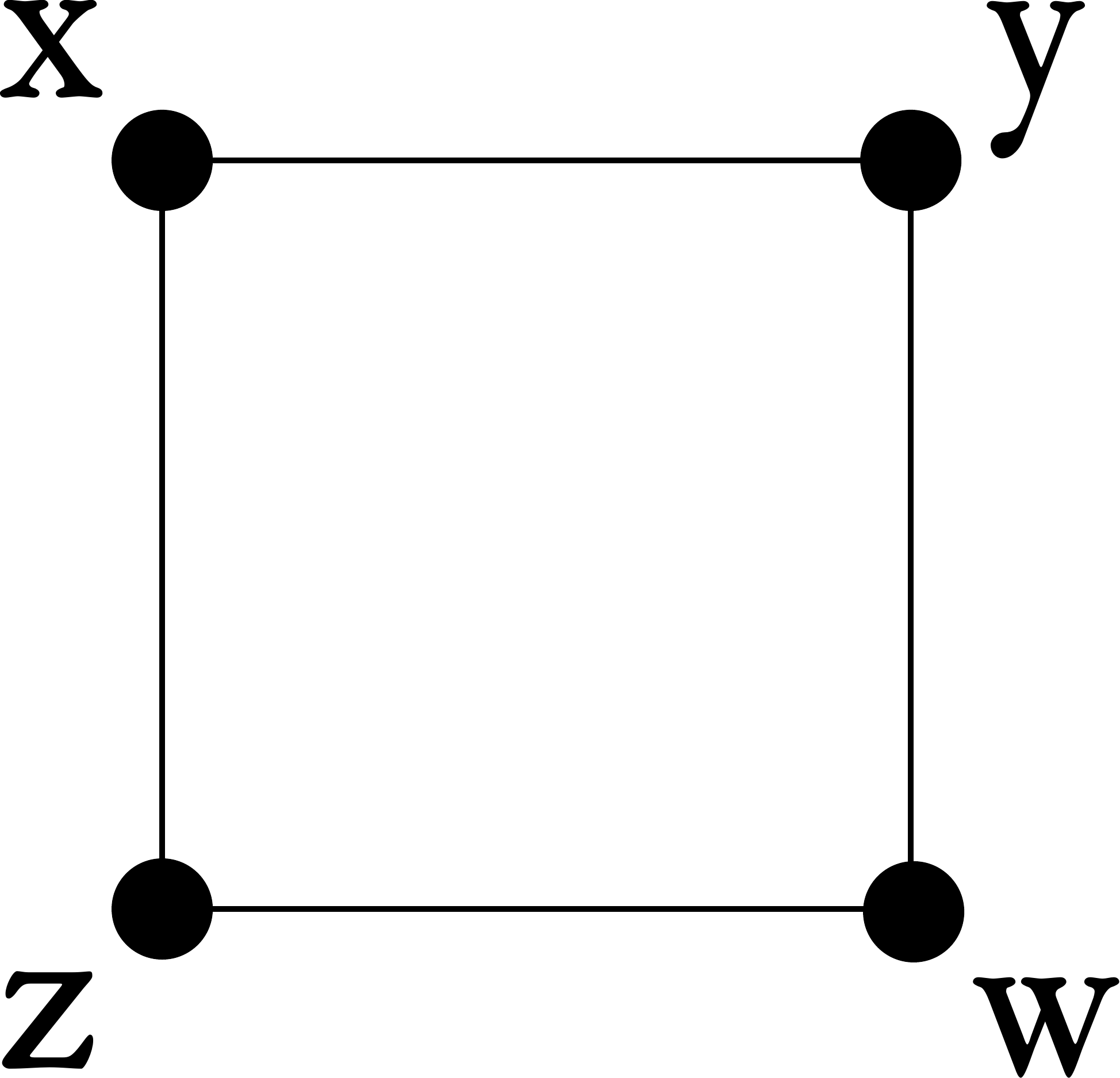}
\caption{$\widetilde{A}_3$ Coxeter diagram}
\label{diagram}
\end{figure}

\begin{lemma}
\label{squarelemma}
If a $4$-cycle with no diagonals in the $1$-skeleton of an $\widetilde{A}_3$ Coxeter complex has a type $z$ vertex, then the three remaining vertices are two type $y$ vertices and a type $z$ vertex. Furthermore, there exists an edge of type $xw$ whose endpoints are adjacent to all vertices of the $4$-cycle.
\end{lemma}

\begin{proof}
We use our description of the 1-skeleton in Theorem \ref{1skeleton}. By symmetry of the 1-skeleton, we can assume without loss of generality that the given $z$ vertex in the lemma is the origin and thus has the ladder of constant values (all balls on one rung). With some abuse of notation, we name the origin $z$ (which also happens to have type $z$), one adjacent vertex $y$, the other adjacent vertex $y'$, and the opposite vertex $a$. 

Consider now the ladders $L_z$, $L_y$, $L_{y'}$, and $L_a$ associated to $z$, $y$, $y'$, and $a$, respectively. Moving from $L_z$ to $L_y$ we increase a set of coordinates $S_y$. Similarly, moving from $L_z$ to $L_{y'}$ we increase a set of coordinates~$S_{y'}$. 

\textbf{Claim.} The assumption that the 4-cycle has no diagonals implies that $(S_y\cup S_{y'})^c$, $S_y-S_{y'}$, $S_{y'}-S_y$, and $S_y\cap S_{y'}$ are all nonempty. 

Indeed, the vertices $y$ and $y'$ are not joined by an edge if and only if $S_y-S_{y'}$ and $S_{y'}-S_y$ are nonempty. Furthermore, for $z$ and its opposite $a$ to not be joined by an edge, $a$ must have height two. Below we show that for this to be possible $S_y\cap S_{y'}$ and $(S_y\cup S_{y'})^c$ must be nonempty. 

The ladder $L_a$ must have some nonempty $C\subseteq S_y\cap S_{y'}$ on the third rung, some nonempty~$D$ with $S_y\triangle S_{y'}\cap D=\emptyset$, $D\subseteq (S_y\cup S_{y'})^c$ on the first rung, and $(S_y\cap S_{y'}-C)\cup ((S_y\cup S_{y'})^c-D)$ on the second rung (Figure \ref{4cycle}). 

\begin{figure}[h]
\centering
\includegraphics[height=7cm]{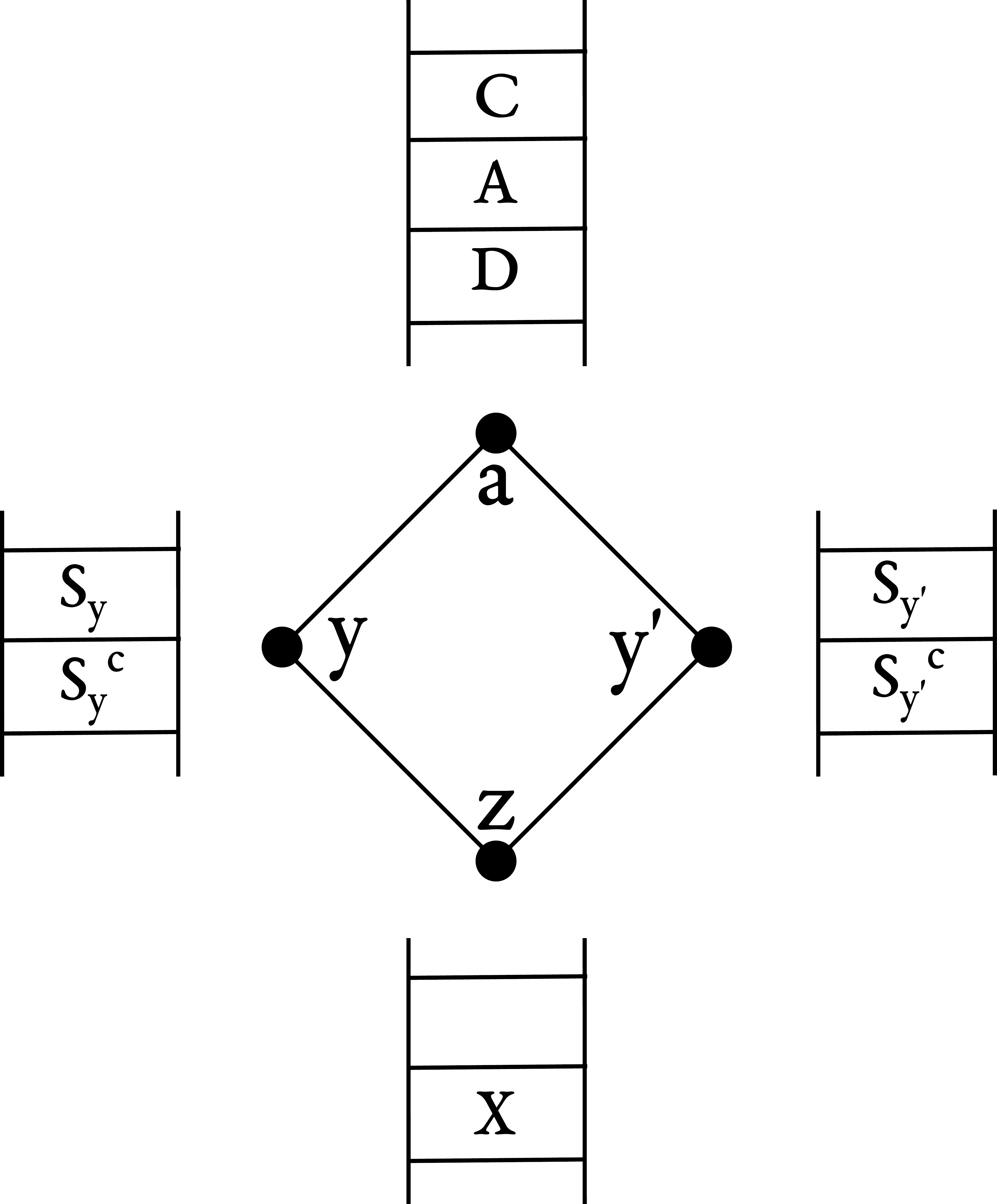}
\caption{$A=(S_y\cap S_{y'}-C)\cup ((S_y\cup S_{y'})^c-D)$}
\label{4cycle}
\end{figure}

We are always guaranteed an edge at the center of such a 4-cycle with endpoints joined to each vertex of the 4-cycle: The endpoints of the edge are the vertices reached from $L_z$ by increasing one of $S_y\cap S_{y'}$ or $S_y\cup S_{y'}$, as seen in Figure \ref{xwedge}. 

\begin{figure}[h]
\centering
\includegraphics[height=8cm]{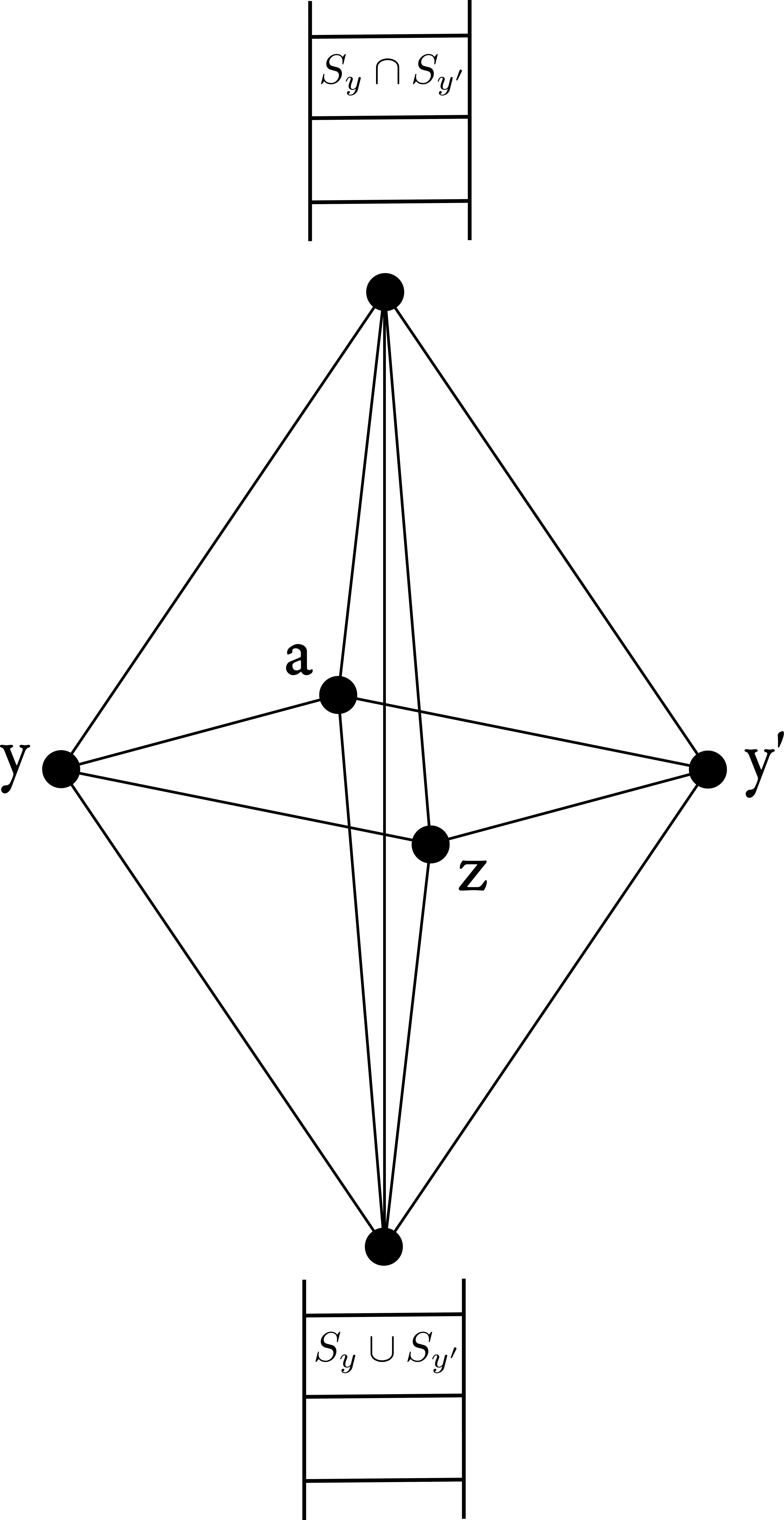}
\caption{The $xw$ edge}
\label{xwedge}
\end{figure}

Now we see that the $z$ and $a$ vertices are the vertices in two 3-simplices glued along their faces opposite the $z$ and $a$ vertices. Thus $a$ has type $z$. The same argument shows the other two vertices $y$ and $y'$ of the 4-cycle must be the same type. Also, they must correspond to a generator whose product with $z$ is order four, since we see a cycle of four simplices glued along their faces. Thus, the vertices $y$ and $y'$ must have type $y$. Furthermore, since vertices joined by an edge cannot have the same type, the only possible types for the endpoints of the central edge are $x$ and $w$. 
\end{proof}

\begin{proof}[Proof of Theorem \ref{wmbuildings}]
As a consequence of Theorem \ref{localglobal}, it suffices to prove local weak modularity. The triangle property given in $\ref{weakly modular}$.i is immediate: The set-up for the property is a vertex $a$, with two adjacent vertices $b$ and $c$ at distance 2 from $a$ (Figure \ref{trianglesetup}). Since $bc$ is an edge, there is an apartment containing the cells $a$ and $bc$. We can then apply weak modularity of the apartment to get a common neighbor of $a$, $b$, and $c$. 

\begin{figure}[h]
\centering
\includegraphics[height=3cm]{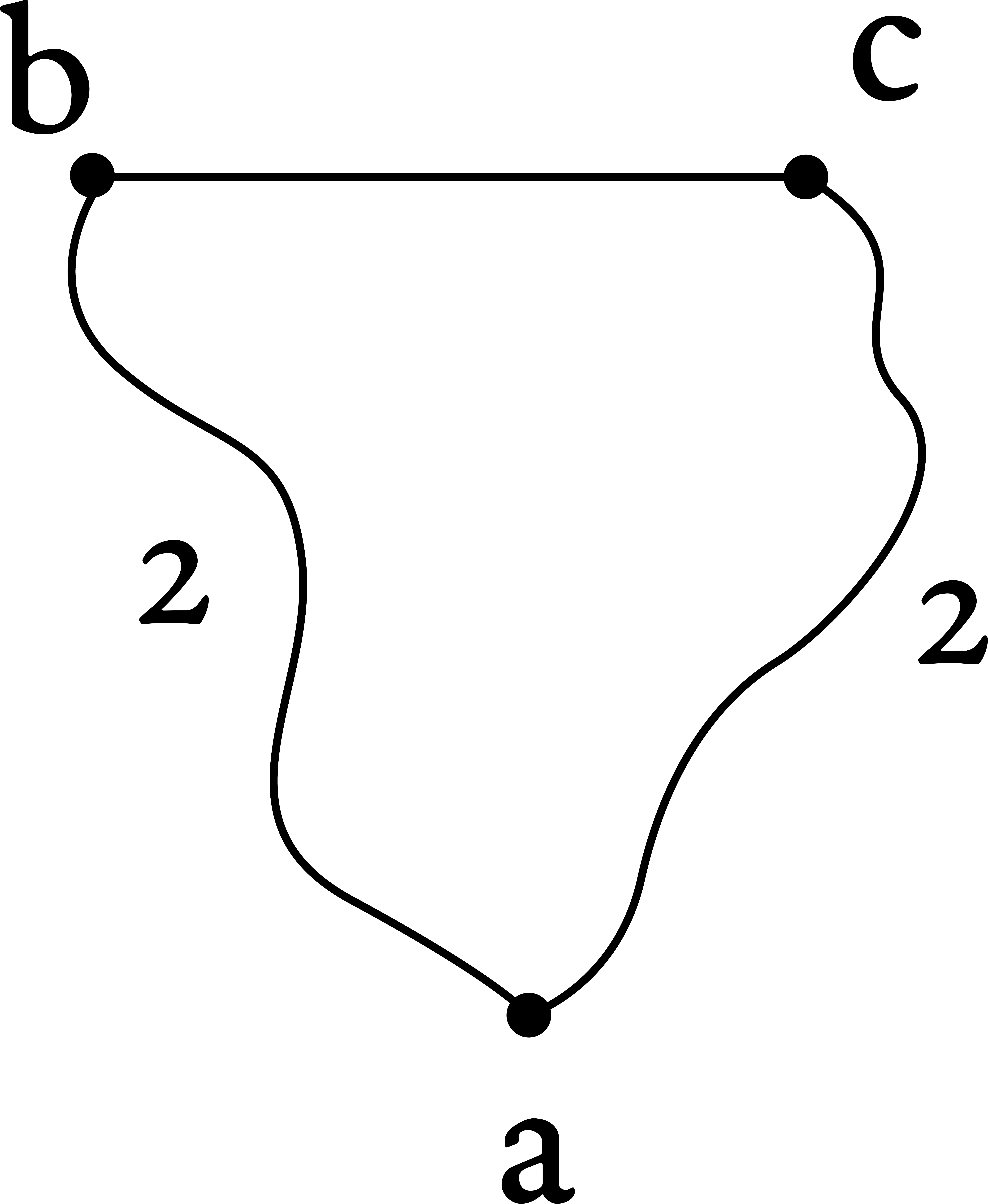}
\caption{Triangle property set-up}
\label{trianglesetup}
\end{figure}

Now consider the set-up for $\ref{weakly modular}$.ii, the quadrangle property. We are given a vertex $a$, with vertices $c$ and $d$ distance two from $a$, and a common neighbor $b$ of $c$ and $d$ at distance 3 from $a$ (Figure \ref{quadranglesetup}). We pick a length 2 path from $a$ to $d$, with an intermediate vertex $f$. We also pick a length two path from $a$ to $c$ with an intermediate vertex~$e$. 

\begin{figure}[h]
\centering
\includegraphics[height=3.5cm]{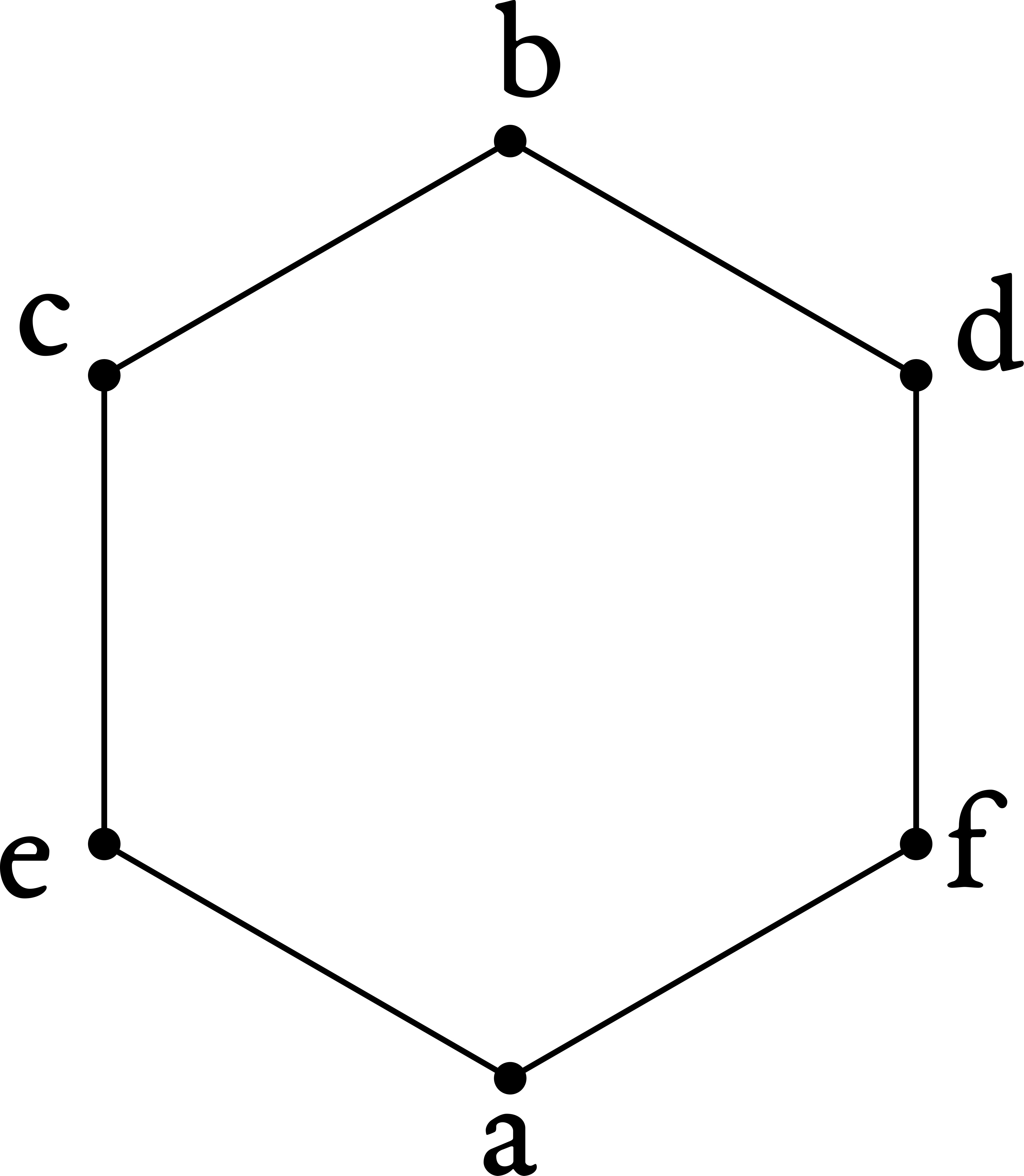}
\caption{Quadrangle property set-up}
\label{quadranglesetup}
\end{figure}

Without loss of generality, we can assume vertex $d$ has type $z$. We proceed by examining the link $lk(d)$ inside some apartment $\Sigma$ containing the edges $bd$ and $df$, pictured in Figure \ref{linkd}.

\begin{figure}[h]
\centering
\includegraphics[height=3.5cm]{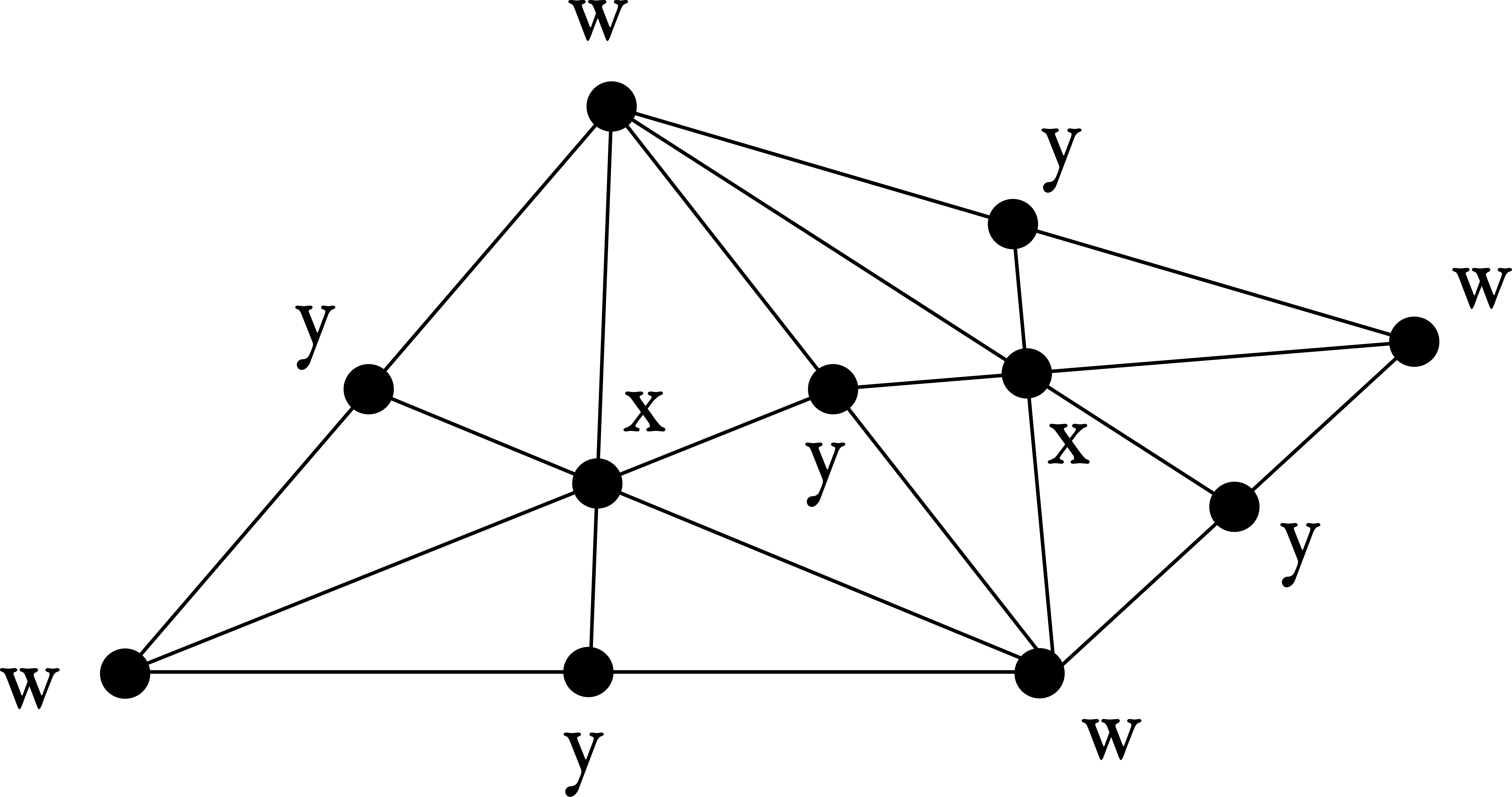}
\caption{The link of $d$}
\label{linkd}
\end{figure}

Next, we consider the realization of $\Sigma$ inside $\mathbb{R}^3$ where generators of $\widetilde{A}_3$ act by reflections through faces of a fixed 3-simplex. We can compute the dihedral angle between faces of the simplex by considering the order of compositions of reflections. The dihedral angle between two faces meeting at an edge $e$ is $\pi$ divided by the order of the product of reflections through the faces. For example, consider the two faces of a simplex containing an edge $yw$. The $yw$ edge has opposite vertices $x$ and $z$ and the action of $x$ and $z$ reflect through the faces meeting along $yw$. Since $xz$ has order three the dihedral angle between the 2-simplices is $\pi/3$. This observation allows us to compute dihedral angles between faces in $lk(d)$, the subcomplex of $\Sigma$ consisting of those faces opposite $d$ of simplices containing $d$. For example, any 2-simplices meeting along an $xw$ edge in $lk(d)$ are the boundary faces of two 3-simplices in $St(d)$. These 3-simplices each have dihedral angle $\pi/2$ at the edge $xw$, so the two 2-simplices have dihedral angle $\pi$ in $lk(d)$. 

$St(d)$ is a closed, convex intersection of half-spaces with boundaries the span of the faces of $lk(d)$. The dihedral angle between any two faces is less than $\pi$, except for those half-spaces meeting along $xw$ edges, which meet at angle $\pi$. The vertices $b$ and $f$ are some pair of vertices on the boundary of $St(d)$. Supposing $b$ and $f$ are not vertices opposite a common $xw$ edge, the $\mathrm{CAT}(0)$ geodesic connecting $b$ and $f$ intersects the interior of $St(d)$. Since apartments are convex subcomplexes, this implies that any apartment containing $b$ and $f$ also contains $d$. In particular, if $\Sigma$ is an apartment containing $bc$ and $af$, then $\Sigma$ also contains $d$. The quadrangle property in $\Sigma$ then provides us with a vertex adjacent to $c$, $d$, and $a$. The only case where the $\mathrm{CAT}(0)$ geodesic joining $b$ and $f$ does not intersect the interior of $St(d)$ is when $b$ and $f$ are opposite a common $xw$ edge. In this case, we know $b$ and $f$ are both type $y$ since $d$ is type $z$. 



Thus we can assume that $a$ and $d$ are joined only by vertices of type $y$. Otherwise, we could replace $f$ with a vertex of type not $y$ and the above argument completes the quadrangle property.

\begin{figure}[h]
\centering
\includegraphics[height=3.5cm]{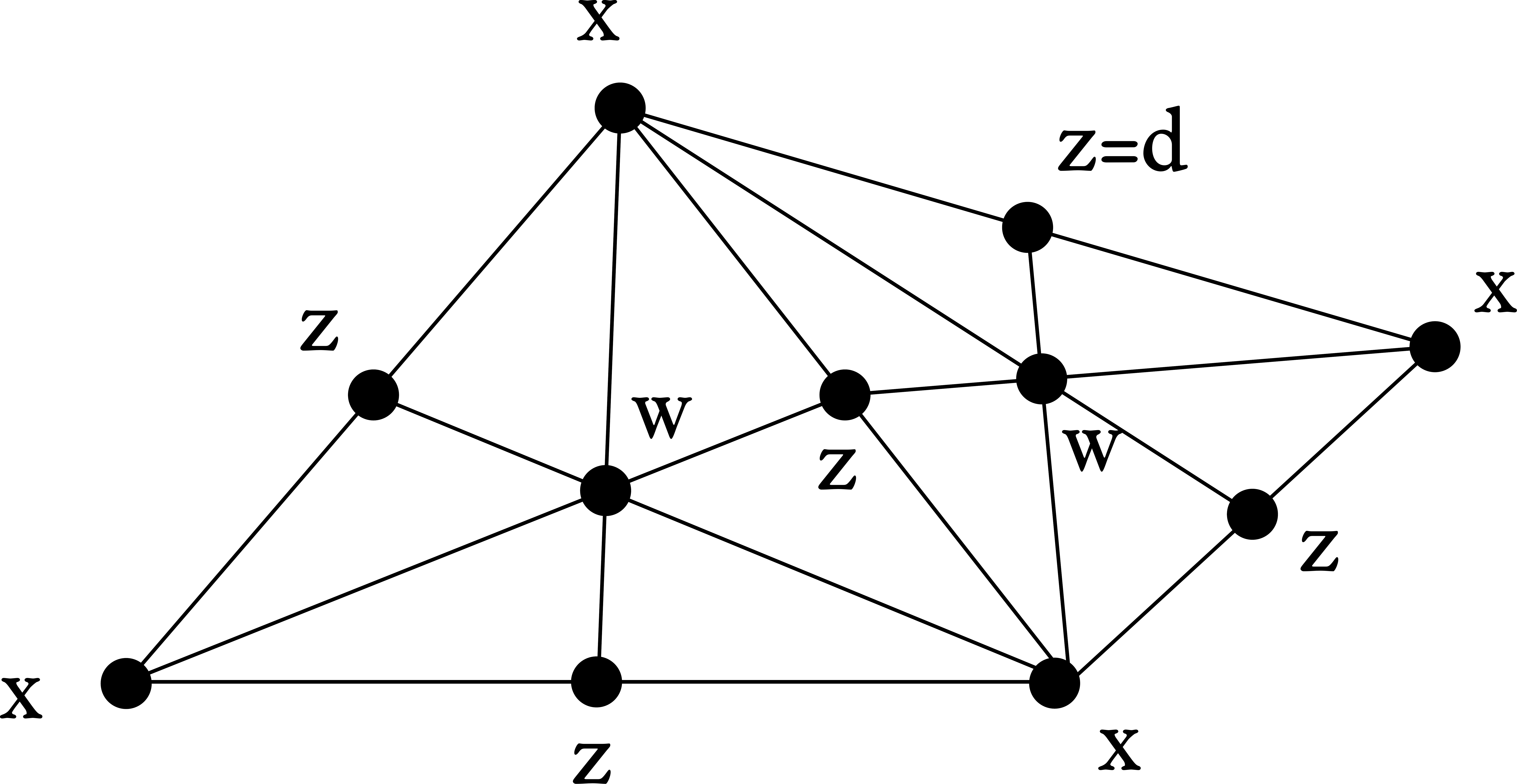}
\caption{The link of $f$}
\label{linkf}
\end{figure}

Now consider $lk(f)$ in some apartment containing $df$ and $af$ (Figure \ref{linkf}). Recall $d\in lk(f)$ is a type $z$ vertex. By inspection, we see $a$ must be the $z$ vertex on the opposite side of $lk(f)$, since all other vertices are either adjacent to $d$ or joined to $d$ by a vertex of type not $y$. Thus the combinatorial geodesic $dfa$ is also a $\mathrm{CAT}(0)$ geodesic in our building. 

By symmetry, we can assume that $e$ is a type $y$ vertex and that $cea$ is a $\mathrm{CAT}(0)$ geodesic. In particular, this implies $c$ is a type $z$ vertex. We have reduced to the case of a 6-cycle with alternating $z$ and $y$ types, as in Figure \ref{6cycle}. 

\begin{figure}[h]
\centering
\includegraphics[height=3.5cm]{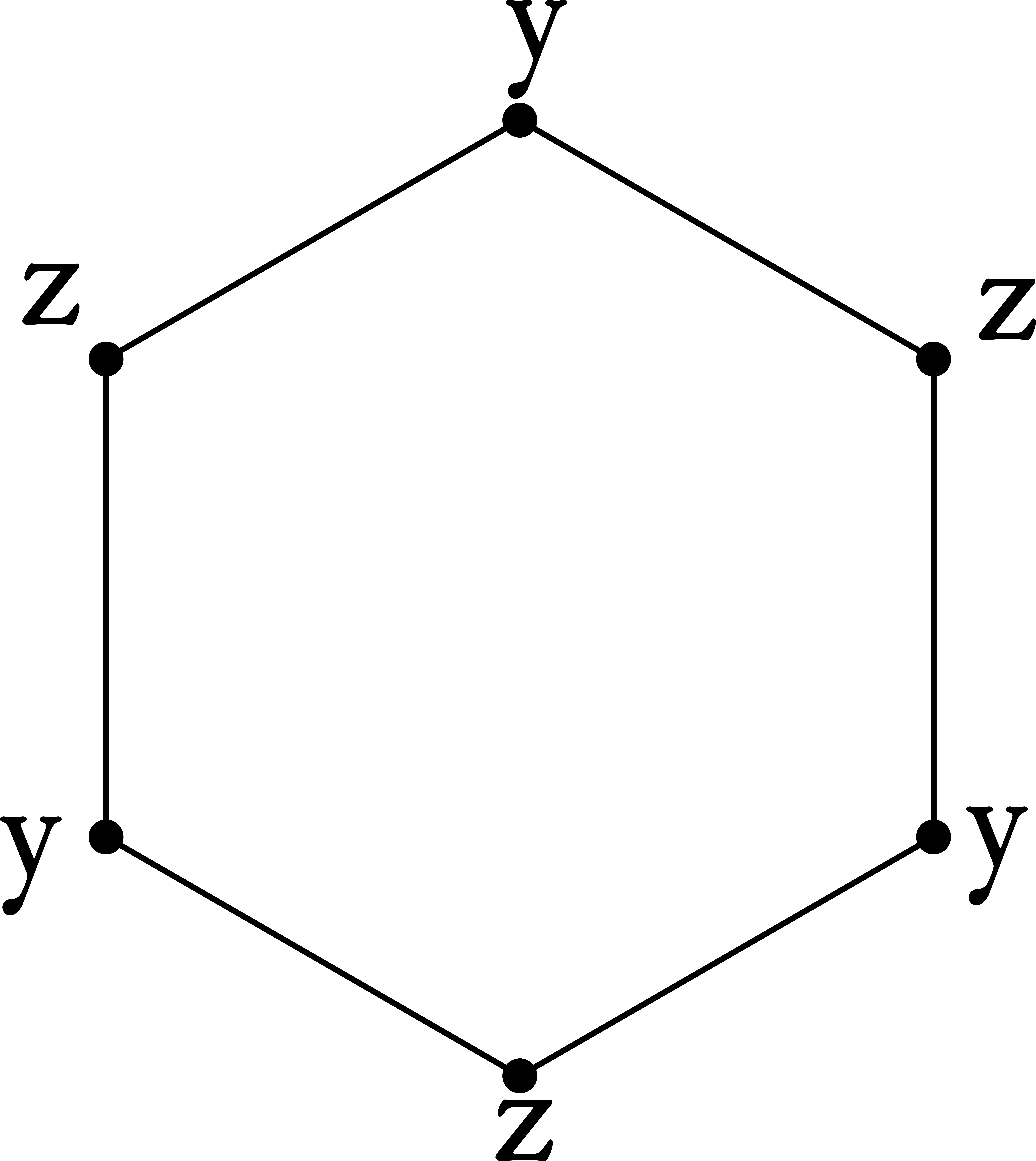}
\caption{6-cycle with alternating types}
\label{6cycle}
\end{figure}

Now, if we invert the diagram (i.e. consider $b$ to be our initial vertex and $e$, $f$, and $a$ to be the vertices relevant for the quadrangle property), then we can repeat our above argument. There are two possible outcomes. We either complete the quadrangle property and find a vertex adjacent to $b$, $e$, and $f$; or we get that $ecb$ and $fdb$ are $\mathrm{CAT}(0)$ geodesics. 

In the case that $ecb$ and $fdb$ are $\mathrm{CAT}(0)$ geodesics, we get that $bcea$ and $bdfa$ are $\mathrm{CAT}(0)$ geodesics and our diagram collapses, i.e. $c=d$ and $e=f$, in which case $e$ is a vertex completing the quadrangle property. 


\begin{figure}[h]
\centering
\includegraphics[height=4cm]{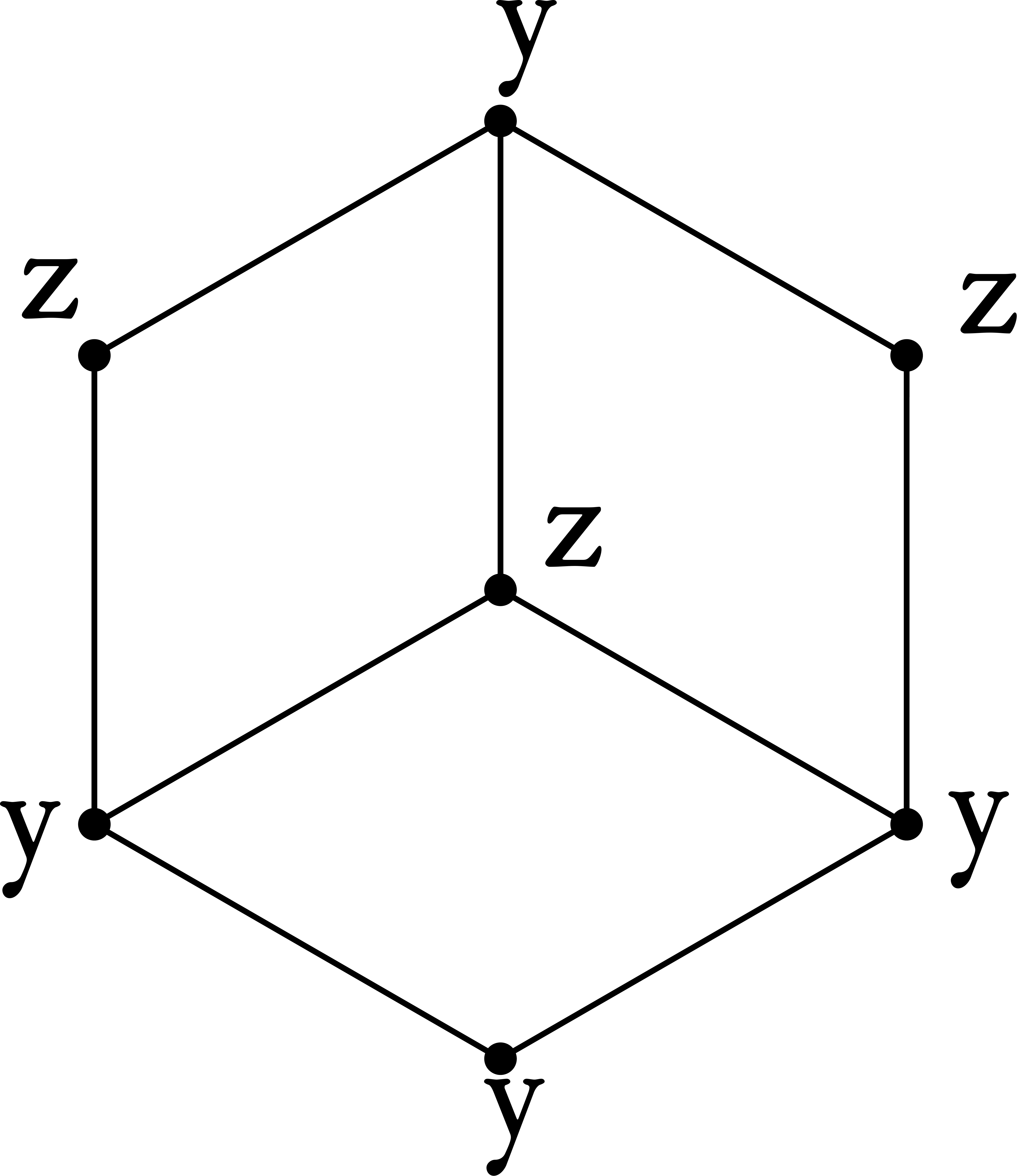}
\caption{Inverted quadrangle property}
\label{invertedquadrangle}
\end{figure}

In the case that the quadrangle property is completed in the inverted diagram, we get a vertex adjacent to $b$, $e$, and $f$. This separates our diagram into three 4-cycles (Figure \ref{invertedquadrangle}). Lemma \ref{squarelemma} implies the central vertex has type $z$ and each 4-cycle has an $xw$ edge in its center, with each endpoint adjacent to each vertex of the 4-cycle. We call the type $x$ vertices $x_1$, $x_2$, and $x_3$. We call the type $w$ vertices $w_1$, $w_2$, and $w_3$, where $w_i$ is adjacent to $x_i$ for $i=1,2,3$. 

Consider the case in which two of these vertices are the same, say $x_1=x_2$. $St(x_1)$ is a convex subcomplex of the building. Since $St(x_1)$ contains the two $y$ vertices adjacent to $x_3$ and the geodesic between the two $y$'s intersects the interior of an edge containing $x_3$, we get that $x_3\in St(x_1)$. But this is a contradiction, since there cannot be adjacent vertices of the same type in a Coxeter complex. So $x_1$, $x_2$, and $x_3$ must be distinct. Analogously, we can assume $w_1$, $w_2$, and $w_3$ are distinct. 

\begin{figure}[h]
\centering
\includegraphics[height=6cm]{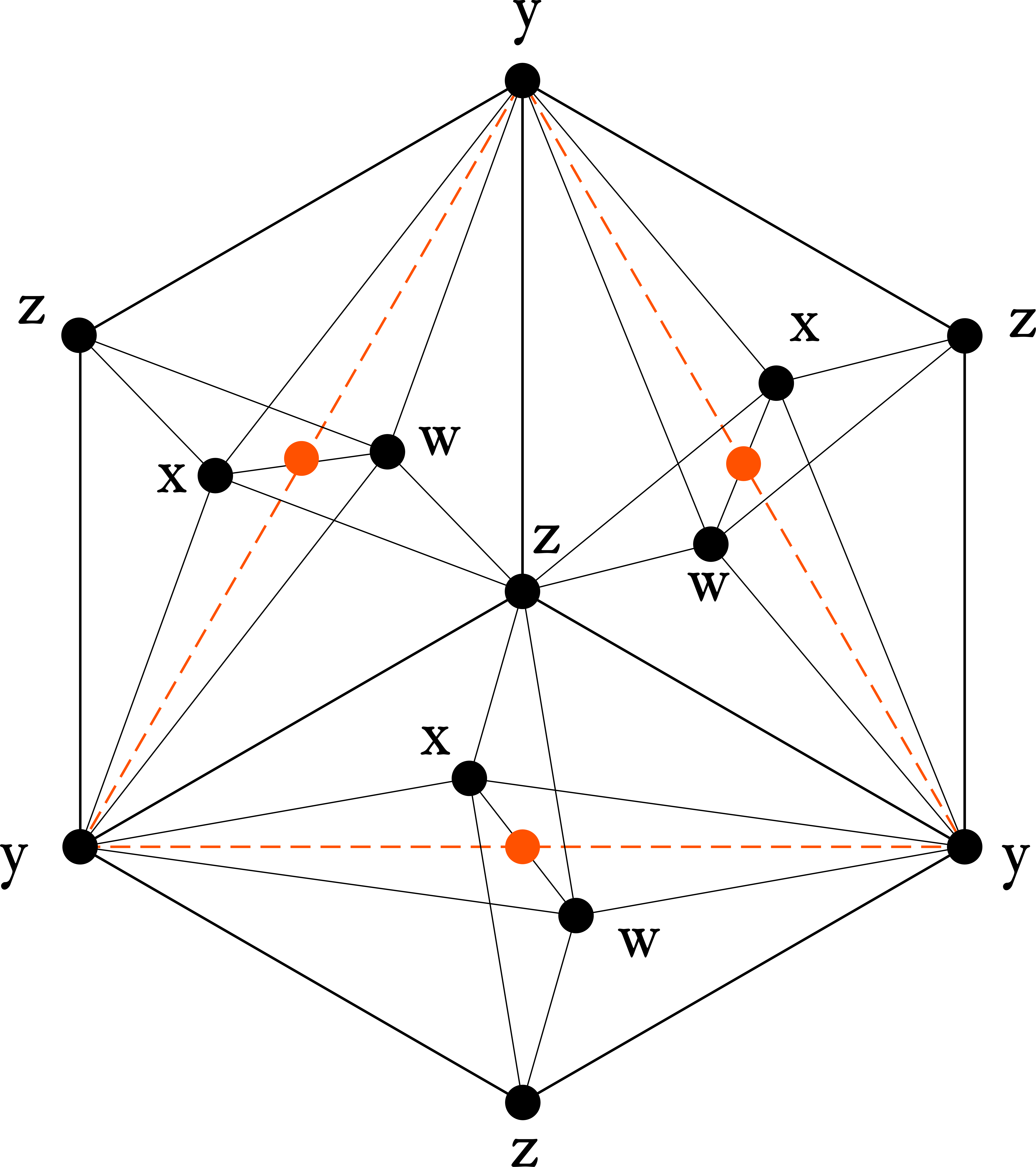}
\caption{A locally geodesic loop}
\end{figure}

Consider now the piecewise geodesic cycle joining each of the three $y$ vertices in the outer 6-cycle to the midpoints of the $x_iw_i$ edges, for $i=1,2,3$. This cycle is made up of six segments. Each segment has length $\pi/4$, since eight such congruent segments make up a geodesic cycle in an $A_3$ Coxeter complex. Any two adjacent geodesic segments are contained in an apartment also containing the central $z$ vertex. The link $lk(z)$ inside such an apartment is an $A_3$ Coxeter complex, which is a sphere tiled by $xyw$ triangles. These triangles have angle $\pi$ at their $y$ vertex and angle $\pi/3$ at their $x$ and $w$ vertices. Adjacent geodesic segments either meet at an $xw$ edge or a $y$ vertex. Since the $xyw$ triangles are isosceles, segments meeting at the midpoint of an $xw$ edge meet at angle $\pi$. Since four $xyw$ triangles meet at a $y$ vertex, segments meeting at $y$ and joining midpoints of opposite $xw$ edges also meet at angle $\pi$. Thus the piecewise geodesic cycle is in fact a locally geodesic cycle. But then we have a locally geodesic cycle of length less than $2\pi$ in $lk(z)$, a $\mathrm{CAT}(1)$ space. This contradiction completes the proof.

\end{proof}

\bibliographystyle{amsalpha}
\providecommand{\bysame}{\leavevmode\hbox to3em{\hrulefill}\thinspace}
\providecommand{\MR}{\relax\ifhmode\unskip\space\fi MR }
\providecommand{\MRhref}[2]{%
  \href{http://www.ams.org/mathscinet-getitem?mr=#1}{#2}
}
\providecommand{\href}[2]{#2}

\end{document}